\documentclass[11pt]{article}
\date{}
\usepackage{amssymb}
\usepackage{amsmath}
\usepackage{amsthm}
\usepackage{graphicx}
\usepackage{indentfirst}
\allowdisplaybreaks[4]
\newtheorem{theorem}{Theorem}

\newtheorem{corollary}[theorem]{Corollary}

\newtheorem{lemma}[theorem]{Lemma}

\newtheorem{remark}[theorem]{Remark}

\numberwithin{equation}{section}
\numberwithin{theorem}{section}
\newcommand{\keywords}[1]{\par\noindent\textbf{Keywords:} #1}
\newcommand{\subjclass}[2]{
  \par\noindent\textbf{Mathematics Subject Classification} #2}

\begin{document}

\title{Liouville type results for the semilinear elliptic equation and inequality on Sasakian manifolds }

\author{ Biqiang Zhao\footnote{
		E-mail addresses: biqiangzhao123@outlook.com}}
            
\maketitle
\setlength{\parindent}{2em}

\begin{abstract}
In this paper, we study the semilinear elliptic equation and inequality on Sasakian manifolds. First, we obtain Liouville theorems for the equation
$ \Delta_b u+F(u)=0$ based on a generalized Jerison-Lee’s formula. Next, we prove the nonexistence of positive weak solutions to the inequality $\Delta_b u+F(u)\leq 0 $ under the volume estimate.

\end{abstract}

\subjclass [{  32V20 · 35J61 · 35R03 · 53C25 }


\section{Introduction}
\label{1}
     In this paper, we study Liouville type theorems for the semilinear equation and differential inequality on Sasakian Manifolds, which generalize the result in \cite{MO} and \cite{GSun}.
     \par
      Recalling the seminal paper \cite{GS}, Gidas and Spruck showed that the only nonnegative solution of the equation
     \begin{align}
         \Delta u+u^\sigma=0\quad \mathrm{in}\ \mathbb{R}^n,\quad n>2
     \end{align}
      is $u\equiv 0$ if $ 1<\sigma<\frac{n+2}{n-2}$. Moreover, Gidas and Spruck also studied the following general equation
      \begin{align}
          \Delta u+F(u)=0\quad \mathrm{in}\ \mathbb{R}^n,\quad n>2.
      \end{align}
      If $\sigma= \frac{n+2}{n-2}$, Caffarelli, Gidas and Spruck \cite{CGS} (see also \cite{LZ}) proved that any positive solution to (1.1) is given by 
      \begin{align*}
          u(x)= \left( \frac{1}{a+b|x-x_0|^2}\right)^{\frac{n-2}{2}},
      \end{align*}
      where $a,b>0,\ ab=\frac{1}{n(n-2)} $ and $x_0\in \mathbb{R}^n$. 
       Later, Li and Zhang \cite {LZ} studied the equation (1.2) and proved that: If $F$ is locally bounded in $ (0,\infty)$ and $F(s)s^{-\frac{n+2}{n-2}}$ is non-increasing in $ (0,\infty)$. Let $u$ be a positive classical solution of (1.2) with $n\geq 3$. Then either for some $b>0$
      \begin{align*}
          bu=\left (  \frac{\mu}{1+\mu ^2|x-\bar{x}|^2}\right )^{\frac{n-2}{2}}
      \end{align*}
      for $\mu >0,\bar{x}\in \mathbb{R}^n$ or $u\equiv a$ for some $a > 0$ such that $F(a) = 0$. In \cite{SZ}, Serrin and Zou extended the above considerations to quasilinear elliptic equations. Since then many improvements or generalizations of such Liouville type theorems have been developed to more general elliptic operators on Euclidean spaces or Riemannian manifolds, see e.g. \cite{CM,EP1,EP2,MMP,GSXX,HSW,CMR} and the references therein.
      \par
      Such Liouville type theorems have also been generalized for subelliptic operators on sub-Riemannian manifolds including pseudo-Hermitian manifolds. The outstanding work \cite{JL,JL1,JL2} of Jerison and Lee studied the so-called CR-Yamabe problem and stimulated us to study the following equation 
      \begin{align}
          \Delta_b u+u^\sigma = 0,\ \mathrm{in} \ \mathbb{H}^n.
      \end{align}
      Inspired by the work of Gidas and Spruck, a natural conjecture is that the nonexistence result for (1.3) holds if $1<\sigma<1+\frac{2}{n}$. In \cite{BP}, Birindelli and Prajapat proved the conjecture for the cylindrical solution via the moving plane method. In \cite{Yu}, Yu studied the nonnegative cylindrical solution of the equation $\Delta_b u+F(u)=0$ and established a Liouville type theorem. For the general solution, Xu \cite{Xu} proved that (1.3) possesses no nontrivial nonnegative solution providing $n>1 $ and $1<\sigma<\frac{(2n+2)(2n+4)}{(2n+1)^2}$. Motivated by Jerison-Lee’s differential identity, Ma and Ou \cite{MO} gave a complete classification of nonnegative solutions to equation (1.3) when $1<\sigma<1+\frac{2}{n}$. For the critical case $\sigma=1+\frac{2}{n}$, due to the work \cite{JL} of Jerison and Lee, there are positive solutions of (1.3) given by
      \begin{align*}
          \mathcal{U}:=\frac{C}{|t+i|z|^2+z\cdot \mu +\lambda|^n},
      \end{align*}
      for some $C>0,\ \lambda\in \mathbb{C},\ \mu\in \mathbb{C}^n$ such that $ \mathrm{Im}(\lambda)>\frac{|\mu|^2}{4}$. Moreover, Jerison and Lee obtained that $\mathcal{U}$ is the unique positive solution satisfying the finite energy assumption (see also \cite{CLMR}).
      \par
      The first purpose of this paper is to study the nonnegative smooth solution of the equation
      \begin{align}
          \Delta_b u+2n^2F(u)=0,
      \end{align}
      which is defined on a complete Sasakian manifold $M^{2n+1}$. Before introducing the result, we give some definitions. A function $f\in C^0([0,\infty))\cap C^1((0,\infty))$ is called subcritical with exponent $\sigma$ if 
        \begin{align*}
            f(t)\geq 0,\quad \sigma f(t)-tf^{'}(t)\geq 0,\quad   \forall t>0.
        \end{align*}
        It is easy to verify that the following functions are subcritical with exponent $\sigma$
        \begin{align*}
               \quad f(t)=t^\sigma, \quad  f(t)=t^\sigma+t^\kappa, \quad  f(t)=\frac{t^\sigma}{1+t^\kappa},
        \end{align*}
        where $\sigma>\kappa>0 $. Obviously, $f$ is a subcritical function with exponent $\sigma$ implying that $ t^{-\sigma}f(t)$ is non-increasing in $(0,\infty)$. We denote by $ B_R$ the ball of radius $R$ centered at a fixed point with respect to the Carnot-Carath$\acute{\mathrm{e}}$odory distance (see Section 2) and by $V(R)$ the volume of $B_R$. The following theorems generalize the result in \cite{MO}.
        \begin{theorem}
            Let $(M^{ 3}, HM, J,\theta)$ be a complete noncompact Sasakian manifold with nonnegative Tanaka-Webster scalar curvature and assume that $F$ is a subcritical function with exponent $\sigma\in (1,3)$. If there exists a constant $1<\kappa\leq \sigma$ such that
          \begin{align*}
              tF^{'}(t)-\kappa F(t)\geq 0,\ \forall t\geq 0,
          \end{align*}
          then every positive solution to (1.4) is constant.
        \end{theorem}
        In the case $n\geq 2$, we have a similar result under the volume estimate.
        \begin{theorem}
           Let $(M^{ 2n+1}, HM, J,\theta),\ n\geq 2$ be a complete noncompact Sasakian manifold with $Ric_b\geq 0 $ and 
           \begin{align*}
               V(R)\leq C R^{2n+2}
           \end{align*}
           for some $C > 0$ and every $R > 0$ large enough.
           Assume that $F$ is a subcritical function with exponent $\sigma\in (1,1+\frac{2}{n})$. If there exist constants $\frac{4}{n}-3<\kappa\leq \sigma$ and $\lambda>1$ such that
          \begin{align*}
              tF^{'}(t)-\kappa F(t)\geq 0,\ \forall t>0,
          \end{align*}
          and 
          \begin{align}
              \liminf\limits_{t\xrightarrow{} \infty} t^{-\lambda}F(t)>0,
          \end{align}
          then every positive solution to (1.4) is constant.
          \begin{remark}
              If $\kappa>1$ in Theorem 1.2, then condition (1.5) is not necessary.
          \end{remark}
           \end{theorem}
          For Sasakian manifolds, we recall the following volume comparison theorem (cf. \cite{LL}).
         \begin{theorem}
             Let $(M^{ 2n+1,} HM, J,\theta)$ be a complete noncompact Sasakian manifold.
             \\
             (1) If $n=1$ and the Tanaka-Webster scalar curvature is nonnegative, then
             \begin{align*}
                 V(R)\leq C R^{4}, \quad \forall R>1,
             \end{align*}
             for some $C>0.$
             \\
             (2) If $n\geq 2$ and the Tanaka-Webster curvature satisfies
             \begin{align}
                 Ric_b(X,X)\geq R(X,JX,JX,X)\geq 0,\quad \forall X\in HM,
             \end{align}
             then 
             \begin{align*}
               V(R)\leq C R^{2n+2},\quad \forall R>1,
           \end{align*}
           for some $C > 0$.
         \end{theorem}
        From Theorem 1.4, we have the following corollary.
        \begin{corollary}
            Let $(M^{ 2n+1}, HM, J,\theta),\ n\geq 2$ be a complete noncompact Sasakian manifold with curvature condition (1.6). Assuming that $F$ satisfies the same conditions as in Theorem 1.2, then every positive solution to (1.4) is constant.
        \end{corollary}
         \par
         Next, we consider positive weak solutions of the semilinear inequality 
         \begin{align}
             \Delta_b u+F(u)\leq 0,
         \end{align}   
         where $F$ is subcritical. Here we say that $u\in C^1(M)$ is a positive weak solution to the inequality (1.7) if $u>0$ and for every $0\leq \psi\in S^{1,2}(M)\cap L^{\infty}(M)$ with compact support, one has
         \begin{align}
                -\int_M\langle \nabla_b u, \nabla_b \psi \rangle+\int_M F(u)\psi  \leq 0,
            \end{align}
            where $ S^{1,2}(M)$ be the space of functions $f\in L^2(M)$ and $|\nabla_b f|\in L^2(M)$.
         In the Euclidean space $\mathbb{R}^n$, it is well known that there is no positive solution for the inequality
         \begin{align}
             \Delta u +u^\sigma\leq 0,
         \end{align}
         if $1<\sigma\leq \frac{n}{n-2}$ (see \cite{GS,NS}).
         Cheng and Yau \cite{CY} first studied the nonnegative solution in terms of the volume of the geodesic ball. More precisely, they proved that any positive solution to $\Delta u\leq 0$ is constant if $VolB_r(x_0)\leq Cr^2$
         holds for all large enough $r$. Here $VolB_r(x_0)$ is the volume of the geodesic ball of radius $r$ centered at $x_0$. In \cite{GSun}, Grigor’yan and Sun proved a Liouville theorem for nonnegative weak solutions of (1.9) if the volume of geodesic ball satisfies 
         \begin{align*}
             VolB_r(x_0) \leq Cr^{ \frac{2\sigma}{\sigma-1}}\mathrm{ln}^{\frac{1}{\sigma-1}}r
         \end{align*}
         for all large enough $r$. Such Liouville results have also been generalized to sub-Riemannian settings. In \cite{BCC}, Birindelli, Dolcetta and Cutri obtained nonexistence results for positive solutions 
         \begin{align*}
             \Delta_b u+u^\sigma\leq 0, \quad \mathrm{in} \ \mathbb{H}^n,
         \end{align*}
         if $1<\sigma\leq 1+\frac{1}{n}$. Recently, Wang and Zhang \cite{WZ} obtained Liouville type theorems for some semilinear differential inequalities to sub-Riemannian manifolds satisfying the generalized curvature-dimension inequality.
         \par
         In this paper, we want to extend the Grigor’yan and Sun's result to (1.7) on Sasakian manifolds.
         \begin{theorem}
             Let $(M^{ 2n+1}, HM, J,\theta)$ be a complete noncompact Sasakian manifold and assume that $F$ is a subcritical function with exponent $\sigma>1$ and 
             \begin{align*}
                 F(t)>0, \quad \forall t>0.
             \end{align*}
        If $V(R)\leq CR^{\frac{2\sigma}{\sigma-1}} \mathrm{ln}^{\frac{1}{\sigma-1}} R $ holds for large enough $R$, then (1.7) has no  positive weak solution.
         \end{theorem}
         \begin{theorem}
             Let $(M^{ 2n+1}, HM, J,\theta)$ be a complete noncompact Sasakian manifold and $F=0$. If $V(R)\leq CR^2 $ holds for large enough $R$, then any positive weak solution of (1.7) is constant.
         \end{theorem}

         \begin{remark}
             Since the proof does not use the commutation relations, Theorem 1.6 and Theorem 1.7 also hold for the general sub-Riemannian manifold satisfying the volume condition.
         \end{remark}
         
         By Theorem 1.4, we have the following corollary. 
         \begin{corollary}
            Let $(M^{ 2n+1,} HM, J,\theta)$ be a complete noncompact Sasakian manifold and assume that $F$ is a subcritical function with exponent $\sigma\in (1,1+\frac{1}{n}]$. If $F(t)>0 $ for any $ t>0 $ and one of the following holds:
            \\
            (1) $n=1$ and $M$ has nonnegative Tanaka-Webster scalar curvature;
            \\
            (2) $n\geq 2$ and $M$ has curvature condition (1.6),
            \\
             then (1.7) has no positive weak solution.
        \end{corollary}

         $\mathbf{Notation}$ Throughout the paper, we adopt Einstein summation convention over repeated indices. Since the constant $C$ is not important, it may vary at different occurrences.
        
         \section{Preliminaries}
         \subsection{Pseudo-Hermitian geometry}
         In this section, we give a brief introduction to pseudo-Hermitian geometry (cf. \cite{DT,We,Ta} for details). A CR manifold is a smooth manifold $M$ endowed with a complex subbundle $H^{1,0}M $ of $TM\otimes \mathbb{C}$ satisfying \begin{equation}
         H^{1,0}M\cap H^{0,1}M=\{0\} ,\ \  [\Gamma(H^{1,0}M),\Gamma(H^{1,0}M)]\subseteq \Gamma(H^{1,0}M)  , 
         \end{equation}
         where $H^{0,1}M=\overline{H^{1,0}M}$. The bundle $H^{1,0}M $ is called a CR structure on the manifold $M$.  Equivalently, the CR structure can also be described by $(HM,J)$, where $HM=Re\{H^{1,0}M\oplus H^{0,1}M\}$ and $J$ is an almost complex structure on $HM$ such that
         \begin{align*}
             J(X+\overline{X})=\sqrt{-1}(X-\overline{X}),\quad \forall X\in H^{1,0}M.
         \end{align*}
         A CR manifold is said to be of hypersurface type if $dim_RM = 2n+1$ and $dim_CH^{1,0}M = n$. In this paper, we only consider oriented and hypersurface type CR manifolds.
         \par
         The orientation of $M$ implies that there exists a global nowhere vanishing 1-form $\theta$ such that $ker\theta=HM$. Any such section $\theta$ is called a pseudo-Hermitian structure on $M$. The Levi form $L_\theta$ of $\theta$ is defined by
          \begin{eqnarray}
            L_\theta(X,Y)=d\theta(X,JY) \nonumber
          \end{eqnarray}
         for any $X,Y\in HM$. $(M,HM,J)$ is said to be strictly pseudoconvex, if the Levi form $L_\theta$ is positive definite for some choice of $\theta$. Then the quadruple $(M,HM,J,\theta)$ is referred to as a pseudo-Hermitian manifold. In this case, it is easy to see that $\theta$ is a contact form and $G_\theta=d\theta(\cdot,J\cdot) $ defines a $J$-invariant Riemannian metric on $HM.$ Then the Reeb vector field $\xi$ is defined by 
         \begin{eqnarray}
           \theta(\xi)=1,\ d\theta(\xi,\cdot)=0,
         \end{eqnarray}
        and $TM=HM \oplus R\xi$. This induces a natural projection $\pi_b:TM\rightarrow HM$. In terms of $\theta$ and the decomposition $TM$, it is natural to define a Riemannian metric
        \begin{eqnarray}
         g_\theta=G_\theta+\theta\otimes \theta ,\nonumber
        \end{eqnarray}
        which is called the Webster metric. The volume form of $g_\theta$ is defined by $dV_\theta=\theta\wedge(d\theta)^{n}/n!$, and is often omitted when we integrate. By requiring $J\xi = 0$, one may extend the complex structure $J$ on $TM$.
        \par
        For a pseudo-Hermitian manifold $(M^{2n+1}, HM, J, \theta)$, a Lipschitz curve $\gamma:[0, l] \rightarrow  M$ is called horizontal if $\gamma^{'}\in H_{\gamma(t)}M$ a.e. in $[0, l]$. Due to the work of Chow \cite{Ch} and Rashevsky \cite{Ra}, there always exist horizontal curves joining two points $p, q\in M. $ The Carnot-Carath$\acute{\mathrm{e}}$odory distance is defined as
        \begin{eqnarray}
              d(p,q)= inf\left\{ \int_0^l \sqrt{g_\theta(\gamma^{'},\gamma^{'})} dt \ | \  \gamma \in \Gamma(p,q) \right\},
            \nonumber
           \end{eqnarray}
        where $\Gamma(p,q) $ denotes the set of all horizontal curves joining $p$ and $q$. Clearly, Carnot-Carath$\acute{\mathrm{e}}$odory is indeed a distance function and induces the same topology on $(M,HM,J,\theta). $ One can refer to \cite{ABB} and \cite{St} and references therein for details.
        \par
        For a pseudo-Hermitian manifold $(M^{2n+1}, HM, J, \theta)$, there exists a canonical connection $\nabla$, called the Tanaka-Webster connection, preserving the horizontal bundle, the CR structure and the Webster metric. Moreover, its torsion $T_\nabla$ satisfies
        \begin{align*}
            T_\nabla(X,Y)=2d\theta(X,Y)\xi\ and\ T_\nabla(\xi,JX)+JT_\nabla(\xi,X)=0
        \end{align*}
        for any $X,Y\in HM.$ The pseudo-Hermitian torsion of $\nabla$, denoted by $\tau $, is defined by $ \tau(X)=T_\nabla(\xi,X)$ for any $X\in TM.$ It is well known that $M$ is Sasakian if $\tau=0. $
        \par
        Assume that $\{\eta_\alpha\}_{\alpha=1}^n$ is a local unitary frame field of $H^{1,0}M$ and let $\{ \theta^\alpha\}_{\alpha=1}^n$ be the dual basis. Then
        \begin{align*}
            d\theta=2\sqrt{-1} \theta^\alpha\wedge \theta^{\Bar{\alpha}}.
        \end{align*}
        For a smooth function $f$ on the Sasakian manifold $(M^{2n+1}, HM, J, \theta)$, its differential $df$ and covariant derivatives $ \nabla df$ can be expressed as
        \begin{align*}
            df=&f_0\theta+f_\alpha\theta^\alpha+f_{\bar{\alpha}}\theta^{\bar{\alpha}},
            \\
            \nabla d f=& f_{\alpha\beta}\theta^\alpha\otimes\theta^\beta +f_{\alpha\bar{\beta}}\theta^\alpha \otimes \theta^{\bar{\beta}}+ f_{\bar{\alpha}\beta}\theta^{\bar{\alpha}}\otimes \theta^\beta+f_{\bar{\alpha}\bar{\beta}}\theta^{\bar{\alpha}}\otimes \theta^{\bar{\beta}}
            \\       &+f_{0\alpha}\theta\otimes\theta^\alpha+f_{0\bar{\alpha}}\theta\otimes\theta^{\bar{\alpha}}+f_{\alpha0}\theta^\alpha\otimes\theta+f_{\bar{\alpha}0}\theta^{\bar{\alpha}}\otimes \theta ,\nonumber
        \end{align*}
        where $f_0=\nabla_{\xi}f,f_\alpha=\nabla_{\eta_\alpha}f,f_{\bar{\alpha}}=\nabla_{\eta_{\bar{\alpha}}}f$ and so on. Then the horizontal gradient of $f$ is given by
        \begin{align*}
            \nabla_bf=f_\alpha \eta_{\bar{\alpha}}+f_{\bar{\alpha}}\eta_\alpha.
        \end{align*}
        Then we have the following communication relations (cf. \cite{DT})
        \begin{align*}
            &f_{\alpha\beta}-f_{\beta\alpha}=0,\quad  f_{\alpha\bar{\beta}}-f_{\bar{\beta}\alpha}=2\sqrt{-1}f_0\delta_\alpha^\beta,
            \\
            &f_{0\alpha}-f_{\alpha 0}=0,\quad f_{\alpha\beta0}-f_{\alpha0\beta}=f_{\alpha\bar{\beta}0}-f_{\alpha0\bar{\beta}}=0,
            \\
            &f_{\alpha\beta\bar{\gamma}}-f_{\alpha\bar{\gamma}\beta}=2\sqrt{-1}\delta_\gamma^\beta f_{\alpha0}+R_{\bar{\mu}\alpha\beta\bar{\gamma}}f_\mu,\quad f_{\alpha\bar{\beta}\bar{\gamma}}-f_{\alpha\bar{\gamma}\bar{\beta}}=0,
        \end{align*}
        where $R_{\bar{\mu}\alpha\beta\bar{\gamma}} $ are the components of the curvature of the Tanaka-Webster connection. Moreover, $R_{\alpha\bar{\beta}} = R_{\bar{\mu}\alpha\mu\bar{\beta}}$ are the components of the pseudo-Hermitian Ricci curvature $Ric_b$ and $R = R_{\alpha\bar{\alpha}}$ denotes the Tanaka-Webster scalar curvature. Finally, we define the horizontal energy density by
        \begin{align*}
            |\nabla_b f|^2=2|\partial f |^2=2f_\alpha f_{\bar{\alpha}}
        \end{align*}
         and the sub-Laplacian of $f$ by
        \begin{align*}
            \Delta_b f=f_{\alpha\bar{\alpha}}+f_{\bar{\alpha}\alpha}.
        \end{align*}

        \subsection{ Jerison-Lee type inequality}
        In this section, we want to establish a generalized inequality for the positive solution $u$ of (1.4). Assume that $F$ is a subcritical function with exponent $\sigma\in (1,1+\frac{2}{n})$. Let $e^f=u^{\frac{1}{n}}$, then we have
        \begin{align*}
            -\Delta_b f= 2n|\partial f|^2+2n\frac{F(u)}{u}=2n|\partial f|^2+2n\frac{F(e^{nf})}{e^{nf}}.
        \end{align*}
        Setting $H(t)=\frac{F(e^{nt})}{e^{nt}}$, then
        \begin{align}
            -\Delta_b f= 2n|\partial f|^2+2nH(f).
        \end{align}
        Moreover, the fact that $F$ is a subcritical function with exponent $\sigma\in (1,1+\frac{2}{n})$ shows that 
        \begin{align}
            (p+2)H-H^{'}\geq 0,
        \end{align}
        where $-2<p=n(\sigma-1)-2<0$. As in \cite{JL}, we define the following tensors
        \begin{align}
            &D_{\alpha\beta}=f_{\alpha\beta}-2f_\alpha f_\beta, \quad D_\alpha=D_{\alpha\beta}f_{\bar{\beta}},\nonumber
            \\
            &E_{\alpha\bar{\beta}}=f_{\alpha\bar{\beta}}-\frac{1}{n}f_{\gamma\bar{\gamma}}\delta_\alpha^\beta,\quad E_\alpha=E_{\alpha\bar{\beta}}f_\beta,
            \\
            &G_\alpha=\sqrt{-1}f_{0\alpha}-\sqrt{-1}f_0f_\alpha+H(f)f_\alpha+|\partial f|^2f_\alpha.\nonumber
        \end{align}
        Introduce the function $g=|\partial f|^2+H(f)-\sqrt{-1}f_0$. Then the equation (2.3) can be written as 
        \begin{align}
            f_{\alpha\bar{\alpha}}=-ng.
        \end{align}
        Direct calculations show that
        \begin{align}
        &E_{\alpha\bar{\beta}}=f_{\alpha\bar{\beta}}+g\delta_\alpha^\beta,\quad E_\alpha=f_{\alpha\bar{\beta}}f_\beta+gf_\alpha,
            \nonumber \\
            &D_\alpha=f_{\alpha\bar{\beta}}f_\beta-2|\partial f|^2f_\alpha, \quad G_\alpha=\sqrt{-1}f_{0\alpha}+gf_\alpha.
        \end{align}
        Moreover, we have
        \begin{align}
            |\partial f |_{\bar{\alpha}}^2=& D_{\bar{\alpha}}+E_{\bar{\alpha}}+\bar{g}f_{\bar{\alpha}}-2f_{\bar{\alpha}}H(f),
            \nonumber \\
            g_{\bar{\alpha}}=& D_{\bar{\alpha}}+E_{\bar{\alpha}}+G_{\bar{\alpha}}+(H^{'}(f)-2H(f))f_{\bar{\alpha}},
            \\
            \bar{g}_\alpha=&D_\alpha+E_\alpha+G_\alpha+(H^{'}(f)-2H(f))f_{{\alpha}},\nonumber
            \\
            g_0=&\sqrt{-1}f_{\alpha}G_{\bar{\alpha}}-\sqrt{-1}f_{\bar{\alpha}}G_\alpha+2f_0|\partial f|^2+H^{'}(f)f_0-\sqrt{-1}f_{00}. \nonumber
        \end{align}
        We are now ready to establish the following inequality, which generalizes the equation (3.5) in \cite{MO} (see also \cite{CMRW}).
        \begin{lemma}
            Let $(M^{ 2n+1}, HM, J,\theta)$ be a complete noncompact Sasakian manifold with $Ric_b\geq 0 $. Setting 
            \begin{align}
                \mathcal{M}=&Re\ \nabla_{\eta_{\bar{\alpha}}}\{e^{2(n-1)f}[(g+3\sqrt{-1}f_0)E_\alpha+(g-\sqrt{-1}f_0)D_\alpha-3\sqrt{-1}f_0G_\alpha \nonumber
                \\
                &-\frac{p}{4}f_\alpha |\partial f|^4] \},
            \end{align}
            then
            \begin{align}
                \mathcal{M}\geq &(|D_{\alpha\beta}|^2+|E_{\alpha\bar{\beta}}|^2)e^{2(n-1)f}H(f)+(|G_\alpha|^2+|D_{\alpha\beta}f_{\bar{\gamma}}+E_{\alpha\bar{\gamma}}f_\beta |^2)e^{2(n-1)f} \nonumber
                \\
                &+s_0(|D_\alpha+G_\alpha|^2+|E_\alpha-G_\alpha|^2)e^{2(n-1)f} \nonumber
                \\
                &+e^{2(n-1)f}|\sqrt{1-s_0}(D_\alpha+G_\alpha)+\frac{1}{2\sqrt{1-s_0}}f_\alpha(H^{'}(f)-2H(f)-\frac{p}{2}|\partial f|^2)|^2 \nonumber
                \\
                &+e^{2(n-1)f}|\sqrt{1-s_0}(E_\alpha-G_\alpha)+\frac{1}{2\sqrt{1-s_0}}f_\alpha(H^{'}(f)-2H(f)-\frac{p}{2}|\partial f|^2)|^2 \nonumber
                \\
                &-(H^{'}(f)-2H(f))e^{2(n-1)f}|\partial f|^2 \left ((2n-1)H(f)+\frac{H^{'}(f)-2H(f)}{2(1-s_0)}\right) \nonumber
                \\
                &-\left[ (H^{'}(f)-2H(f))(2n-1-\frac{p}{2(1-s_0)})-\frac{p}{4}(n+2)H(f)\right]e^{2(n-1)f}|\partial f|^4 \nonumber
                \\
                &-\frac{p}{4}(n+\frac{p}{2(1-s_0)})|\partial f|^6-3n (H^{'}(f)-2H(f))e^{2(n-1)f}f_0^2,
            \end{align}
            where $0<s_0<1$ is a constant.
        \end{lemma}
        \begin{proof}
            Let $L=L_1+L_2+L_3+L_4$, where
            \begin{align*}
                &L_1=\nabla_{\eta_{\bar{\alpha}}}\{(g+3\sqrt{-1}f_0)E_\alpha e^{2(n-1)f}\},\quad L_2=\nabla_{\eta_{\bar{\alpha}}}\{(g-\sqrt{-1}f_0)D_\alpha e^{2(n-1)f}\},
                \\
                &L_3=\nabla_{\eta_{\bar{\alpha}}}\{-3\sqrt{-1}f_0G_\alpha e^{2(n-1)f}\},\quad L_4=\nabla_{\eta_{\bar{\alpha}}}\{-\frac{p}{4}f_\alpha |\partial f|^4 e^{2(n-1)f}\}.
            \end{align*}
            In the computation of $L_1 $ and $L_3$, the communication relations are the same as in the Heisenberg group $\mathbb{H}^n$. Hence, as done in \cite{MO}, we have
            \begin{align}
                e^{-2(n-1)f}L_1=&(g+3\sqrt{-1}f_0)|E_{\alpha\bar{\beta}}|^2+(g_{\bar{\alpha}}-3G_{\bar{\alpha}})E_\alpha \nonumber
                \\
                &+(3\bar{g}+2(n-1)(g+3\sqrt{-1}f_0))f_{\bar{\alpha}}E_\alpha \nonumber
                \\
                &+(1-n)(g+3\sqrt{-1}f_0)f_\alpha g_{{\bar{\alpha}}},
                \\
                e^{-2(n-1)f}L_3=&3|G_\alpha|^2-3(\bar{g}+2(n-1)\sqrt{-1}f_0)f_{\bar{\alpha}}G_\alpha-3\sqrt{-1}f_0f_\alpha g_{\bar{\alpha}} \nonumber
                \\
                &-3nf_0g_0+3n\sqrt{-1}f_0|g|^2+6nf_0^2g.
            \end{align}
            Next we compute $L_2$. By the communication relations, we obtain
            \begin{align*}
                f_{\alpha\beta\bar{\alpha}}=&f_{\alpha\bar{\alpha}\beta}+2\sqrt{-1}\delta_\alpha^\beta f_{0\alpha}+R_{\beta\bar{\mu}}f_\mu 
                \\
                =&(f_{\bar{\alpha}\alpha}+2n\sqrt{-1}f_0)_\beta+2\sqrt{-1}f_{0\beta}+R_{\beta\bar{\mu}}f_\mu 
                \\
                =&-n\bar{g}_\beta+2(n+1)(G_\beta-gf_\beta) +R_{\beta\bar{\mu}}f_\mu
                \\
                =&2(n+1)G_\beta-n\bar{g}_\beta-2(n+1)f_\beta g  +R_{\beta\bar{\mu}}f_\mu.
            \end{align*}
           Combining with (2.5) and (2.6), we have
           \begin{align}
               D_{\alpha\bar{\alpha}}=&f_{\alpha\beta\bar{\alpha}}f_{\bar{\beta}}+f_{\alpha\beta}f_{\bar{\beta}\bar{\alpha}}-2|\partial f|^2_{\bar{\alpha}}f_\alpha-2|\partial f|^2f_{\alpha\bar{\alpha}} \nonumber
               \\
               =& ( 2(n+1)G_\beta-n\bar{g}_\beta-2(n+1)f_\beta g +R_{\beta\bar{\mu}}f_\mu)f_{\bar{\beta}}\nonumber
               \\
               &+(D_{\alpha\beta}+2f_\alpha f_\beta )(D_{\bar{\alpha}\bar{\beta}}+2f_{\bar{\alpha}}f_{\bar{\beta}})\nonumber
               \\
               &-2(D_{\bar{\alpha}}+E_{\bar{\alpha}}+\bar{g}f_{\bar{\alpha}}-2f_{\bar{\alpha}}H(f) )f_\alpha+2n|\partial f|^2g \nonumber
               \\
               =&|D_{\alpha\beta}|^2+2f_{\bar{\alpha}}D_\alpha-2f_\alpha E_{\bar{\alpha}}+2(n+1)f_{\bar{\alpha}}G_\alpha-nf_{\bar{\alpha}}\bar{g}_\alpha+R_{\beta\bar{\mu}}f_\mu f_{\bar{\beta}}.
           \end{align}
           It follows that 
           \begin{align}
               &e^{-2(n-1)f}L_2=e^{-2(n-1)f}\nabla_{\eta_{\bar{\alpha}}}\{(g-\sqrt{-1}f_0)D_\alpha e^{2(n-1)f}\} \nonumber
               \\
               =&(g-\sqrt{-1}f_0)D_{\alpha\bar{\alpha}}+(g_{\bar{\alpha}}-\sqrt{-1}f_{0\bar{\alpha}})D_\alpha+2(n-1)(g-\sqrt{-1}f_0)f_{\bar{\alpha}}D_\alpha \nonumber
               \\
               =&(g-\sqrt{-1}f_0)(|D_{\alpha\beta}|^2+2f_{\bar{\alpha}}D_\alpha-2f_\alpha E_{\bar{\alpha}}+2(n+1)f_{\bar{\alpha}}G_\alpha-nf_{\bar{\alpha}}\bar{g}_\alpha+R_{\beta\bar{\mu}}f_\mu f_{\bar{\beta}} ) \nonumber
               \\
               &+(g_{\bar{\alpha}}+G_{\bar{\alpha}}-\bar{g}f_{\bar{\alpha}})D_\alpha+2(n-1)(g-\sqrt{-1}f_0)f_{\bar{\alpha}}D_\alpha \nonumber
            \\
             =&(g-\sqrt{-1}f_0)|D_{\alpha\beta}|^2+(g_{\bar{\alpha}}+G_{\bar{\alpha}})D_\alpha+(2ng-\bar{g}-2n\sqrt{-1}f_0)f_{\bar{\alpha}}D_\alpha \nonumber
             \\
             &+(g-\sqrt{-1}f_0)(-2f_\alpha E_{\bar{\alpha}}+2(n+1)f_{\bar{\alpha}}G_\alpha-nf_{\bar{\alpha}}\bar{g}_\alpha+R_{\beta\bar{\mu}}f_\mu f_{\bar{\beta}} ).
           \end{align}
           As for $L_4$, by (2.6) and (2.8), we obtain
           \begin{align}
               e^{-2(n-1)f}L_4=&e^{-2(n-1)f}\nabla_{\eta_{\bar{\alpha}}}\{-\frac{p}{4}f_\alpha |\partial f|^4 e^{2(n-1)f}\} \nonumber
               \\
               =&-\frac{p}{2}|\partial f|^2|\partial f|^2_{\bar{\alpha}}f_\alpha-\frac{p}{4}f_{\alpha\bar{\alpha}}|\partial f|^4-\frac{p}{2}(n-1)|\partial f|^6 \nonumber
               \\
               =&-\frac{p}{2}(D_{\bar{\alpha}}+E_{\bar{\alpha}})|\partial f|^2f_\alpha-\frac{p}{4}n|\partial f|^6
               \\
               &+\frac{p}{4}(n+2)|\partial f|^4H(f)-\frac{p}{4}(n+2)\sqrt{-1}f_0|\partial f|^4. \nonumber
           \end{align}
           From (2.11)-(2.15) and $ f_\alpha E_{\bar{\alpha}}=f_{\bar{\alpha}}E_\alpha$, we have
           \begin{align}
               &e^{-2(n-1)f}L
               \nonumber\\
               =&(g-\sqrt{-1}f_0)|D_{\alpha\beta}|^2+(g+3\sqrt{-1}f_0)|E_{\alpha\bar{\beta}}|^2+3|G_\alpha|^2+(g_{\bar{\alpha}}+G_{\bar{\alpha}})D_\alpha \nonumber
               \\
               &+(g_{\bar{\alpha}}-3G_{\bar{\alpha}})E_\alpha +(2ng-\bar{g}-2n\sqrt{-1}f_0)f_{\bar{\alpha}}D_\alpha-\frac{p}{2}|\partial f|^2f_\alpha D_{\bar{\alpha}} \nonumber
               \\
               &+(2(n-2)g+3\bar{g}-\frac{p}{2}|\partial f|^2+(6n-4)\sqrt{-1}f_0)f_\alpha E_{\bar{\alpha}}\nonumber
               \\
               &+(2(n+1)g-3\bar{g}-(8n-4)\sqrt{-1}f_0)f_{\bar{\alpha}}G_\alpha \nonumber
               \\
               &-((n-1)g+3n\sqrt{-1}f_0)f_\alpha g_{\bar{\alpha}}-n(g-\sqrt{-1}f_0)f_{\bar{\alpha}}\bar{g}_\alpha-3nf_0g_0 \nonumber
               \\
               &+3n\sqrt{-1}f_0|g|^2+6n|f_0|^2g-\frac{p}{4}n|\partial f|^6+\frac{p}{4}(n+2)|\partial f|^4H(f) \nonumber
               \\
               &-\frac{p}{4}(n+2)\sqrt{-1}f_0|\partial f|^4+(g-\sqrt{-1}f_0)R_{\alpha\bar{\beta}}f_{\bar{\alpha}}f_\beta.
           \end{align}        
        Substituting (2.8) into (2.16), we obtain
        \begin{align}
            &e^{-2(n-1)f}L
               \nonumber\\
               =&(g-\sqrt{-1}f_0)|D_{\alpha\beta}|^2+(g+3\sqrt{-1}f_0)|E_{\alpha\bar{\beta}}|^2+3|G_\alpha|^2+(D_{\bar{\alpha}}+E_{\bar{\alpha}}+2G_{\bar{\alpha}})D_\alpha \nonumber
               \\
               &+(D_{\bar{\alpha}}+E_{\bar{\alpha}}-2G_{\bar{\alpha}})E_\alpha-(2n-1)|\partial f |^2[H^{'}(f)-2H(f)]H(f)
               \nonumber\\
               &-[(2n-1)(H^{'}(f)-2H(f) )-\frac{p}{4}(n+2)H(f)]|\partial f|^4-\frac{p}{4}n|\partial f|^6\nonumber
               \\
               &-3nf_0^2(H^{'}(f)-2H(f) )+(g-\sqrt{-1}f_0)R_{\alpha\bar{\beta}}f_{\bar{\alpha}}f_{\beta}+\mathcal{Q} \nonumber
               \\
               &-\frac{p}{4}(n+2)\sqrt{-1}f_0|\partial f|^4-\sqrt{-1}f_0(H^{'}(f)-2H(f) )|\partial f|^2+3n\sqrt{-1}f_0|g|^2\nonumber
               \\
               &-6n\sqrt{-1}f_0^3+3n\sqrt{-1}f_0f_{00},
        \end{align}
        where 
        \begin{align*}
            \mathcal{Q}=&[(n-1)g-(n+2)\sqrt{-1}f_0+H^{'}(f)-2H(f)]f_{\bar{\alpha}}D_\alpha
            \\
            &-[(n-1)g+3n\sqrt{-1}f_0+\frac{p}{2}|\partial f|^2]f_\alpha D_{\bar{\alpha}}
            \\
            &+[(4n+2)\sqrt{-1}f_0-\frac{p}{2}|\partial f|^2+H^{'}(f)-2H(f)]f_{\bar{\alpha}}E_\alpha
            \\
            &+[(n-1)g-(4n+2)\sqrt{-1}f_0]f_{\bar{\alpha}}G_\alpha+[(1-n)g-6n\sqrt{-1}f_0]f_\alpha G_{\bar{\alpha}}
            \\
            =&(f_{\bar{\alpha}}D_\alpha+f_{\bar{\alpha}}E_\alpha)(H^{'}(f)-2H(f)-\frac{p}{2}|\partial f|^2 )
            \\
            &+(f_{\bar{\alpha}}D_\alpha-f_\alpha D_{\bar{\alpha}})(( n-1+\frac{p}{2})|\partial f|^2+(n-1)H(f))
            \\
            &+(2n+1)\sqrt{-1}f_0(f_{\bar{\alpha}}D_\alpha+f_\alpha D_{\bar{\alpha}})+(4n+2)\sqrt{-1}f_0f_{\bar{\alpha}}E_\alpha
            \\
            &+(n-1)(f_{\bar{\alpha}}G_\alpha-f_\alpha G_{\bar{\alpha}} )(|\partial f|^2+H(f))-5(n+1)\sqrt{-1}f_0(f_{\bar{\alpha}}G_\alpha+f_\alpha G_{\bar{\alpha}}) .
        \end{align*}
        Considering the real part of $L$ and using the curvature condition, we have
        \begin{align*}
            \mathcal{M}=&(|D_{\alpha\beta}|^2+|E_{\alpha\bar{\beta}}|^2)H(f)e^{2(n-1)f}
            \\
            &+(|G_\alpha|^2+|D_\alpha+G_\alpha|^2+|E_\alpha-G_\alpha|^2+|D_{\alpha\beta}f_{\bar{\gamma}}+E_{\alpha\bar{\gamma}}f_\beta |^2)e^{2(n-1)f}
            \\
            &+(H^{'}(f)-2H(f)-\frac{p}{2}|\partial f|^2 )Re(f_{\bar{\alpha}}D_\alpha+f_{\bar{\alpha}}E_\alpha)e^{2(n-1)f}
            \\
            &-(2n-1)[H^{'}(f)-2H(f)]H(f)|\partial f |^2e^{2(n-1)f}-\frac{p}{4}n|\partial f|^6e^{2(n-1)f}
            \\
            &-[(2n-1)(H^{'}(f)-2H(f) )-\frac{p}{4}(n+2)H(f)]|\partial f|^4e^{2(n-1)f}
            \\
            &-3nf_0^2(H^{'}(f)-2H(f) )e^{2(n-1)f}+(|\partial f |^2+H(f))e^{2(n-1)f}R_{\alpha\bar{\beta}}f_{\bar{\alpha}}f_\beta 
            \\
            \geq & (|D_{\alpha\beta}|^2+|E_{\alpha\bar{\beta}}|^2)H(f)e^{2(n-1)f}+(|G_\alpha|^2+|D_{\alpha\beta}f_{\bar{\gamma}}+E_{\alpha\bar{\gamma}}f_\beta |^2)e^{2(n-1)f}
            \\
            &+(s_0+1-s_0)(|D_\alpha+G_\alpha|^2+|E_\alpha-G_\alpha|^2)e^{2(n-1)f}
            \\
            &+(H^{'}(f)-2H(f)-\frac{p}{2}|\partial f|^2 )Re(f_{\bar{\alpha}}(D_\alpha+G_\alpha))e^{2(n-1)f}
            \\
            &+(H^{'}(f)-2H(f)-\frac{p}{2}|\partial f|^2 )Re(f_{\bar{\alpha}}(E_\alpha-G_\alpha))e^{2(n-1)f}
            \\
            &-(2n-1)[H^{'}(f)-2H(f)]H(f)|\partial f |^2e^{2(n-1)f}-\frac{p}{4}n|\partial f|^6e^{2(n-1)f}
            \\
            &-[(2n-1)(H^{'}(f)-2H(f) )-\frac{p}{4}(n+2)H(f)]|\partial f|^4e^{2(n-1)f}
            \\
            &-3nf_0^2(H^{'}(f)-2H(f) )e^{2(n-1)f}.
        \end{align*}
        Finally, we rewrite the right hand side of the above inequality and obtain
        \begin{align*}
            \mathcal{M}\geq &(|D_{\alpha\beta}|^2+|E_{\alpha\bar{\beta}}|^2)H(f)e^{2(n-1)f}+(|G_\alpha|^2+|D_{\alpha\beta}f_{\bar{\gamma}}+E_{\alpha\bar{\gamma}}f_\beta |^2)e^{2(n-1)f}
            \\
            &+s_0(|D_\alpha+G_\alpha|^2+|E_\alpha-G_\alpha|^2)e^{2(n-1)f}
            \\
            &+e^{2(n-1)f}|\sqrt{1-s_0}(D_\alpha+G_\alpha)+\frac{1}{2\sqrt{1-s_0}}f_\alpha(H^{'}(f)-2H(f)-\frac{p}{2}|\partial f|^2)|^2 \nonumber
                \\
                &+e^{2(n-1)f}|\sqrt{1-s_0}(E_\alpha-G_\alpha)+\frac{1}{2\sqrt{1-s_0}}f_\alpha(H^{'}(f)-2H(f)-\frac{p}{2}|\partial f|^2)|^2
                \\
            &-\frac{1}{2(1-s_0)}e^{2(n-1)f}|\partial f|^2(H^{'}(f)-2H(f)-\frac{p}{2}|\partial f|^2)^2
            \\
            &-(2n-1)[H^{'}(f)-2H(f)]H(f)|\partial f |^2e^{2(n-1)f}-\frac{p}{4}n|\partial f|^6e^{2(n-1)f}
            \\
            &-[(2n-1)(H^{'}(f)-2H(f) )-\frac{p}{4}(n+2)H(f)]|\partial f|^4e^{2(n-1)f}
            \\
            &-3nf_0^2(H^{'}(f)-2H(f) )e^{2(n-1)f}.
        \end{align*}
        Inequality (2.10) follows by manipulating the last four lines in the above.
        \end{proof}

        \section{Proof of Theorem 1.1 and Theorem 1.2} 
        Let $(M^{ 2n+1}, HM, J,\theta)$ be a complete noncompact Sasakian manifold. As in Sect 2.2, we assume that $u$ is a positive solution of (1.4) and $e^f=u^{\frac{1}{n}}$. Moreover, given $R>0$, there exists a cut-off function $\varphi$ such that 
        $$  \left\{
        \begin{array}{rcl}
         &\varphi=1    &  in\ B_R, \\
         &0\leq \varphi\leq 1,    & in\ B_{2R} ,\\
         &\varphi=0,   &  in\ M \backslash B_{2R},  \\
          & |\partial \varphi|\leq \frac{C}{R},  & in\ M. 
        \end{array}
        \right.
        $$
        In preparation for the proof, we estimate $\mathcal{M}$ from below by appropriate nonnegative expressions, thereby guaranteeing positive coefficients. As a preparation for the proof, we need to bound $\mathcal{M}$ by suitable nonnegative terms to guarantee positive coefficients. In fact, since $-2<p<0$ and $(p+2)H-H^{'}\geq 0,$ 
        we have
        \begin{align*}
            &(H^{'}(f)-2H(f))\left(2n-1-\frac{p}{2(1-s_0)}\right)-\frac{p}{4}(n+2)H(f)
            \\
            \leq &p\left(2n-1-\frac{p}{2(1-s_0)}\right)H(f)-\frac{p}{4}(n+2)H(f)
            \\
            = &p\left(\frac{7n-6}{4}-\frac{p}{2(1-s_0)}\right)H(f)\leq 0.
        \end{align*}
        Hence, we only need to estimate 
        \begin{align*}
            -(H^{'}(f)-2H(f))e^{2(n-1)f}|\partial f|^2 \left ((2n-1)H(f)+\frac{H^{'}(f)-2H(f)}{2(1-s_0)}\right)
        \end{align*}
        and
        \begin{align*}
            -\frac{p}{4}\left(n+\frac{p}{2(1-s_0)}\right)|\partial f|^6.
        \end{align*}
        More precisely, we need to find a suitable $s_0$ such that 
        \begin{align*}
           (2n-1)H(f)+\frac{H^{'}(f)-2H(f)}{2(1-s_0)}\geq 0,\quad and \quad n+\frac{p}{2(1-s_0)}>0.
        \end{align*}
       Equivalently, 
       \begin{align}
           H^{'}(f)\geq 2(2-2n+(2n-1)s_0)H(f),\quad and \quad 0<s_0<1+\frac{p}{2n}.
       \end{align}
       Note that $H(t)=\frac{F(e^{nt})}{e^{nt}}$. Therefore, under the condition in Theorem 1.1 or 1.2, there exists a constant $\kappa\in (\frac{4}{n}-3,1+\frac{2}{n})$ such that 
       \begin{align*}
           tF^{'}(t)-\kappa F(t)\geq 0,\quad \forall t>0.
       \end{align*}
       Then we have
       \begin{align*}
           H^{'}(f)\geq(n\kappa-n)H(f)\geq  2(2-2n)H(f).
       \end{align*}
       In this case, we can find a small $0<s_0<min\{1+\frac{p}{2n}, \frac{n\kappa-4+3n}{4n-2} \}$ such that (3.1) holds. In particular, from $(p+2)H-H^{'}\geq 0$ and (2.10), we have
       \begin{align}
                \mathcal{M}\geq &(|D_{\alpha\beta}|^2+|E_{\alpha\bar{\beta}}|^2)e^{2(n-1)f}H(f)+(|G_\alpha|^2+|D_{\alpha\beta}f_{\bar{\gamma}}+E_{\alpha\bar{\gamma}}f_\beta |^2)e^{2(n-1)f} \nonumber
                \\
                &+s_0(|D_\alpha+G_\alpha|^2+|E_\alpha-G_\alpha|^2)e^{2(n-1)f} \nonumber
                \\
                &+e^{2(n-1)f}|\sqrt{1-s_0}(D_\alpha+G_\alpha)+\frac{1}{2\sqrt{1-s_0}}f_\alpha(H^{'}(f)-2H(f)-\frac{p}{2}|\partial f|^2)|^2 \nonumber
                \\
                &+e^{2(n-1)f}|\sqrt{1-s_0}(E_\alpha-G_\alpha)+\frac{1}{2\sqrt{1-s_0}}f_\alpha(H^{'}(f)-2H(f)-\frac{p}{2}|\partial f|^2)|^2 \nonumber
                \\
                &-p \left (2n-1+\frac{n\kappa-n-2}{2(1-s_0)}\right)e^{2(n-1)f}H^2(f)|\partial f|^2 \nonumber
                \\
                &-\left[ p(2n-1-\frac{p}{2(1-s_0)})-\frac{p}{4}(n+2)\right]e^{2(n-1)f}H(f)|\partial f|^4 \nonumber
                \\
                &-\frac{p}{4}(n+\frac{p}{2(1-s_0)})|\partial f|^6-3n pe^{2(n-1)f}H(f)f_0^2,
            \end{align}
            where $0<s_0<min\{1+\frac{p}{2n}, \frac{n\kappa-4+3n}{4n-2} \}$. Note that all coefficients are positive in (3.2). In the sequel, we assume that $0<s_0<min\{1+\frac{p}{2n}, \frac{n\kappa-4+3n}{4n-2} \}$. Next, we establish two integral estimates, which will be used in the proof.
            \begin{lemma}
            For $s>0$ large enough, we have
                \begin{align}
                    \int_M \varphi^{s-2}f_0^2e^{2(n-1)f}\leq &\epsilon R^2\int_M \varphi^s \mathcal{M}+C\int_M\varphi^{s-2}|\partial f|^4e^{2(n-1)f} \nonumber
                    \\
                    &+C\int_M\varphi^{s-2}|\partial f|^2e^{2(n-1)f}H(f)+\frac{C}{R^2}\int_M\varphi^{s-4}|\partial f|^2e^{2(n-1)f},
                    \\
                    \int_M \varphi^se^{2(n-1)f}H^3(f)\leq &C \int_M \varphi^{s}|\partial f|^2e^{2(n-1)f}H^2(f)+\frac{C}{R^2}\int_M \varphi^{s-2}e^{2(n-1)f}H^2(f),
                \end{align}
                where $\epsilon$ is a small positive constant.
            \end{lemma}
            \begin{proof}
                By the definition of $f$, we have
                \begin{align*}
                    &e^{-2(n-1)f}Re \nabla_{\eta_{\bar{\alpha}}}\{\sqrt{-1}f_0f_\alpha  e^{2(n-1)f}\}
                    \\
                    =&Re\{\sqrt{-1}f_{0\bar{\alpha}}f_\alpha+\sqrt{-1}f_0f_{\alpha\bar{\alpha}} + 2(n-1)\sqrt{-1}f_0|\partial f|^2\}
                    \\
                    =&Re\{(-G_{\bar{\alpha}}+\sqrt{-1}f_0f_{\bar{\alpha}}+H(f)f_{\bar{\alpha}}+|\partial f|^2f_{\bar{\alpha}})f_\alpha-n\sqrt{-1}f_0g \}
                    \\
                    =& -ReG_{\bar{\alpha}}f_\alpha+|\partial f|^2H(f)+|\partial f|^4-nf_0^2.
                \end{align*}
                Multiplying both sides of the above equation by $\varphi^{s-2}e^{2(n-1)f}$, then we have
                \begin{align}
                    &\int_M \varphi^{s-2}Re \nabla_{\eta_{\bar{\alpha}}}\{\sqrt{-1}f_0f_\alpha  e^{2(n-1)f}\}
                    \nonumber\\
                    =&\int_M \varphi^{s-2}(-ReG_{\bar{\alpha}}f_\alpha+|\partial f|^2H(f)+|\partial f|^4-nf_0^2 )e^{2(n-1)f}.
                \end{align}
            Integrating by part and using Young’s inequality, we get the following
            \begin{align*}
                &\int_M \varphi^{s-2} nf_0^2e^{2(n-1)f}
                \\
                \leq & \int_M \varphi^{s-2}(|\partial f|^2H(f)+|\partial f|^4)e^{2(n-1)f}+\int_M \varphi^{s-2}|G_\alpha|\cdot|\partial f|e^{2(n-1)f}
                \\
                &+\frac{C}{R}\int_M \varphi^{s-3}|f_0|\cdot |\partial f|e^{2(n-1)f}
                \\
                \leq & C\int_M \varphi^{s-2}(|\partial f|^2H(f)+|\partial f|^4)e^{2(n-1)f}+\epsilon R^2\int_M \varphi^{s}|G_\alpha|^2e^{2(n-1)f}
                \\
                &+\epsilon \int_M \varphi^{s-2}f_0^2e^{2(n-1)f}+\frac{C}{R^2}\int_M\varphi^{s-4}|\partial f|^2e^{2(n-1)f}.
            \end{align*}
            Since $\mathcal{M}\geq |G_\alpha|^2e^{2(n-1)f}$, we find that 
            \begin{align*}
                \int_M \varphi^{s-2} f_0^2e^{2(n-1)f}\leq &C\int_M \varphi^{s-2}(|\partial f|^2H(f)+|\partial f|^4)e^{2(n-1)f}+\epsilon R^2\int_M \varphi^{s}\mathcal{M}
                \\
                &+\frac{C}{R^2}\int_M\varphi^{s-4}|\partial f|^2e^{2(n-1)f}.
            \end{align*}
            This gives (3.3). The proof of (3.4) is similar. Multiplying both sides of the equation $f_{\alpha\bar{\alpha}}=-ng$ by $\varphi^{s}e^{2(n-1)f}H^2(f)$, we have
            \begin{align*}
                & n\int_M \varphi^sg e^{2(n-1)f}H^2(f)=-\int_M \varphi^s f_{\alpha\bar{\alpha}}e^{2(n-1)f}H^2(f)
                \\
                =&s\int_M \varphi^{s-1}\varphi_{\bar{\alpha}}f_\alpha e^{2(n-1)f}H^2(f)+2(n-1)\int_M \varphi^s|\partial f|^2e^{2(n-1)f}H^2(f)
                \\
                &+2\int_M \varphi^s|\partial f|^2e^{2(n-1)f}H(f)H^{'}(f).
            \end{align*}
            Taking the real parts and using $(p+2)H-H^{'}\geq 0,$ we obtain
            \begin{align*}
                & n\int_M \varphi^s(H(f)+|\partial f|^2 ) e^{2(n-1)f}H^2(f)
                \\
                =&s\int_M \varphi^{s-1}Re(\varphi_{\bar{\alpha}}f_\alpha) e^{2(n-1)f}H^2(f)+2(n-1)\int_M \varphi^s|\partial f|^2e^{2(n-1)f}H^2(f)
                \\
                &+2\int_M \varphi^s|\partial f|^2e^{2(n-1)f}H(f)H^{'}(f)
                \\
                \leq & s\int_M \varphi^{s-1}Re(\varphi_{\bar{\alpha}}f_\alpha) e^{2(n-1)f}H^2(f)+2(n+p+1)\int_M \varphi^s|\partial f|^2e^{2(n-1)f}H^2(f).
            \end{align*}
            Thus 
            \begin{align*}
                \int_M \varphi^s e^{2(n-1)f}H^3(f)\leq & \frac{C}{R}\int_M \varphi^{s-1}|\partial f| e^{2(n-1)f}H^2(f)+C\int_M \varphi^s|\partial f|^2e^{2(n-1)f}H^2(f)
                \\
                \leq & \frac{C}{R^2}\int_M \varphi^{s-2}e^{2(n-1)f}H^2(f)+C\int_M \varphi^s|\partial f|^2e^{2(n-1)f}H^2(f).
            \end{align*}
            This completes the proof.
            \end{proof}
            Now we are ready to prove Theorem 1.1 and Theorem 1.2. We first give a proof of Theorem 1.2.
            
            ~\\
            $\mathbf{Proof\ of\ Theorem\ 1.2.}$  Multiplying $\varphi^s$ on both sides of (2.9) and integrating over $M$ give
            \begin{align*}
                \int_M \varphi^s\mathcal{M}=&\int_M \varphi^s Re \nabla_{\eta_{\bar{\alpha}}}\{e^{2(n-1)f}[(D_\alpha+E_\alpha)(|\partial f|^2+H(f)) \nonumber
                \\
                &-\sqrt{-1}f_0(2D_\alpha-2E_\alpha+3G_\alpha)-\frac{p}{4}f_\alpha |\partial f|^4] \}
                \\
                =&-s\int_M \varphi^{s-1} Re\ \varphi_{\bar{\alpha}} \{e^{2(n-1)f}[(D_\alpha+E_\alpha)(|\partial f|^2+H(f)) \nonumber
                \\
                &\quad -\sqrt{-1}f_0(2D_\alpha-2E_\alpha+3G_\alpha)-\frac{p}{4}f_\alpha |\partial f|^4] \}
                \\
                \leq &\frac{C}{R} \int_M \varphi^{s-1}  \{e^{2(n-1)f}[|D_\alpha+E_\alpha|\cdot (|\partial f|^2+H(f)) \nonumber
                \\
                &\quad +|f_0|\cdot|2D_\alpha-2E_\alpha+3G_\alpha|+|\partial f|^5] \}
                \\
                \leq &\frac{C}{R} \int_M \varphi^{s-1}  \{e^{2(n-1)f}[(|D_\alpha+G_\alpha|+|E_\alpha-G_\alpha|) (|\partial f|^2+H(f)) \nonumber
                \\
                &\quad +|f_0|\cdot(2|D_\alpha+G_\alpha|+2|E_\alpha-G_\alpha|+|G_\alpha|)+|\partial f|^5] \}
                \\
                \leq & \epsilon \int_M \varphi^{s}(|D_\alpha+G_\alpha|^2+|E_\alpha-G_\alpha|^2+|G_\alpha|^2)e^{2(n-1)f}
                \\
                &+\frac{C}{R^2}\int_M \varphi^{s-2}(|\partial f|^4+H^2(f)+f^2_0 )e^{2(n-1)f}+\frac{C}{R}\int_M \varphi^{s-1}|\partial f|^5e^{2(n-1)f},
            \end{align*}
            where $\epsilon$ is a small positive constant. Noticing that the coefficients are positive in (3.2), we obtain
            \begin{align}
                \int_M \varphi^s\mathcal{M}\leq &\frac{C}{R^2}\int_M \varphi^{s-2}(|\partial f|^4+H^2(f)+f^2_0 )e^{2(n-1)f}+\frac{C}{R}\int_M \varphi^{s-1}|\partial f|^5e^{2(n-1)f} \nonumber
                \\
                \leq & \frac{C}{R^2}\int_M \varphi^{s-2}(|\partial f|^4+H^2(f)+ |\partial f|^2H(f))e^{2(n-1)f} 
                \\
                &+\frac{C}{R^4}\int_M \varphi^{s-4}|\partial f|^2e^{2(n-1)f}+\frac{C}{R}\int_M \varphi^{s-1}|\partial f|^5e^{2(n-1)f}. \nonumber
            \end{align}
            Here, we use (3.3) to deal with the term $f_0^2$ in the last equality. Using Young’s inequality again, we have
            \begin{align}
                \frac{1}{R}\int_M \varphi^{s-1}|\partial f|^5e^{2(n-1)f}\leq & \epsilon \int_M \varphi^{s}|\partial f|^6e^{2(n-1)f}+\frac{C}{R^6}\int_M \varphi^{s-6}e^{2(n-1)f},
                \\
                \frac{1}{R^4}\int_M \varphi^{s-4}|\partial f|^2e^{2(n-1)f}\leq &\epsilon \int_M \varphi^{s}|\partial f|^6e^{2(n-1)f}+\frac{C}{R^6}\int_M \varphi^{s-6}e^{2(n-1)f},
                \\
                \frac{1}{R^2}\int_M \varphi^{s-2} |\partial f|^2H(f)e^{2(n-1)f}\leq &\epsilon \int_M \varphi^{s}|\partial f|^4H(f)e^{2(n-1)f}\nonumber
                \\
                &+\frac{C}{R^4}\int_M \varphi^{s-4}e^{2(n-1)f}H(f),
            \end{align}
            and
            \begin{align}
                \frac{1}{R^2}\int_M \varphi^{s-2}|\partial f|^4e^{2(n-1)f}\leq &\epsilon \int_M \varphi^{s}|\partial f|^6e^{2(n-1)f}+\frac{C}{R^6}\int_M \varphi^{s-6}e^{2(n-1)f}.
            \end{align}
            Taking $\epsilon$ small enough and substituting (3.7)-(3.10) into (3.6), we obtain
            \begin{align}
                \int_M \varphi^s\mathcal{M}\leq &\frac{C}{R^6}\int_M \varphi^{s-6}e^{2(n-1)f}+\frac{C}{R^4}\int_M \varphi^{s-4}e^{2(n-1)f}H(f) \nonumber
                \\
                &+\frac{C}{R^2}\int_M \varphi^{s-2}e^{2(n-1)f}H^2(f).
            \end{align}
            From (3.4) and 
            \begin{align*}
                \mathcal{M}\geq C|\partial f|^2e^{2(n-1)f}H^2(f),
            \end{align*}
             we conclude that 
            \begin{align}
                &\int_M \varphi^se^{2(n-1)f}H^3(f)
                \nonumber\\
                \leq &C \int_M \varphi^s\mathcal{M}+\frac{C}{R^2}\int_M \varphi^{s-2}e^{2(n-1)f}H^2(f) \nonumber
                \\
                \leq &\frac{C}{R^6}\int_M \varphi^{s-6}e^{2(n-1)f}+\frac{C}{R^4}\int_M \varphi^{s-4}e^{2(n-1)f}H(f) \nonumber
                \\
                &+\frac{C}{R^2}\int_M \varphi^{s-2}e^{2(n-1)f}H^2(f).
            \end{align}
            According to $\liminf\limits_{t\xrightarrow{} \infty} t^{-\lambda}F(t)>0 $, there $t_0>1$ such that 
            \begin{align*}
                F(t)\geq Ct^\lambda,\quad \forall t\geq t_0.
            \end{align*}
            Without loss of generality, we assume that $1<\lambda<1+\frac{2}{n}.$ Since $F$ is subcritical, then we have
            \begin{align*}
                F(t)\geq \frac{F(t_0)}{t_0^\sigma}t^\sigma\geq Ct^\sigma,\quad \forall 0< t\leq t_0.
            \end{align*}
            By the definition of $H$, there exists a $f_0$ such that 
            \begin{align}
                H(f)=\frac{F(e^{nf})}{e^{nf}}\geq Ce^{n(\lambda-1)f}, \quad \forall f>f_0,
            \end{align}
            and
            \begin{align}
                H(f)=\frac{F(e^{nf})}{e^{nf}}\geq Ce^{n(\sigma-1)f}, \quad \forall f\leq f_0.
            \end{align}
            Choosing $a=n(\lambda-1),\ b=n(\sigma-1) $, then we have $0<a,b<2 $. In the case $f>f_0$, from (3.13), we obtain
            \begin{align}
                e^{2(n-1)f}H^2(f)\leq& C e^{2(n-1)f}H^2(f)(e^{-af}H(f) )^{\frac{2(n-1)}{2(n-1)+3a}}
                \nonumber\\
                \leq &C e^{2(n-1)\frac{2(n-1)+2a}{2(n-1)+3a}f}H^{\frac{6(n-1)+6a}{2(n-1)+3a}}(f),
                \\
                e^{2(n-1)f}H(f)\leq& C e^{2(n-1)f}H(f)(e^{-af}H(f) )^{\frac{4(n-1)}{2(n-1)+3a}}
                \nonumber\\
                \leq &C e^{2(n-1)\frac{2(n-1)+a}{2(n-1)+3a}f}H^{\frac{6(n-1)+3a}{2(n-1)+3a}}(f),
            \end{align}
            and
            \begin{align}
                e^{2(n-1)f}\leq& C e^{2(n-1)f}(e^{-af}H(f) )^{\frac{6(n-1)}{2(n-1)+3a}}
                \nonumber\\
                \leq &C e^{2(n-1)\frac{2(n-1)}{2(n-1)+3a}f}H^{\frac{6(n-1)}{2(n-1)+3a}}(f).
            \end{align}
            From (3.15)-(3.17) and Young's inequality, we have
            \begin{align}
                &\frac{1}{R^2}\int_{\{f>f_0\}} \varphi^{s-2}e^{2(n-1)f}H^2(f)
                \nonumber\\
                \leq& C \int_{\{f>f_0\}} \varphi^{\frac{2(n-1)+2a}{2(n-1)+3a}s}e^{2(n-1)\frac{2(n-1)+2a}{2(n-1)+3a}f}H^{\frac{6(n-1)+6a}{2(n-1)+3a}}(f)\cdot \varphi^{\frac{a}{2(n-1)+3a}s-2}R^{-2} \nonumber
                \\
                \leq& \epsilon \int_{\{f>f_0\}} \varphi^se^{2(n-1)f}H^3(f)+CR^{-2\frac{2(n-1)+3a}{a} }\int_{\{f>f_0\}} \varphi^{ s-2\frac{2(n-1)+3a}{a}},
                \\
                &\frac{1}{R^4}\int_{\{f>f_0\}} \varphi^{s-4}e^{2(n-1)f}H(f)
                \nonumber\\
                \leq& C \int_{\{f>f_0\}} \varphi^{\frac{2(n-1)+a}{2(n-1)+3a}s}e^{2(n-1)\frac{2(n-1)+a}{2(n-1)+3a}f}H^{\frac{6(n-1)+3a}{2(n-1)+3a}}(f)\cdot \varphi^{\frac{2a}{2(n-1)+3a}s-4}R^{-4} \nonumber
                \\
                \leq& \epsilon \int_{\{f>f_0\}} \varphi^se^{2(n-1)f}H^3(f)+CR^{-2\frac{2(n-1)+3a}{a} }\int_{\{f>f_0\}} \varphi^{ s-2\frac{2(n-1)+3a}{a}},             
            \end{align}
            and
            \begin{align}
                &\frac{1}{R^6}\int_{\{f>f_0\}} \varphi^{s-6}e^{2(n-1)f}
                \nonumber\\
                \leq& C \int_{\{f>f_0\}} \varphi^{\frac{2(n-1)}{2(n-1)+3a}s}e^{2(n-1)\frac{2(n-1)}{2(n-1)+3a}f}H^{\frac{6(n-1)}{2(n-1)+3a}}(f)\cdot \varphi^{\frac{3a}{2(n-1)+3a}s-6}R^{-4} \nonumber
                \\
                \leq& \epsilon \int_{\{f>f_0\}} \varphi^se^{2(n-1)f}H^3(f)+CR^{-2\frac{2(n-1)+3a}{a} }\int_{\{f>f_0\}} \varphi^{ s-2\frac{2(n-1)+3a}{a}},     
            \end{align}
            where $\epsilon$ is a small positive constant. In the case $f\leq f_0$, similar to $f>f_0$, we have
            \begin{align}
                 &\frac{1}{R^2}\int_{\{f\leq f_0\}} \varphi^{s-2}e^{2(n-1)f}H^2(f)+\frac{1}{R^4}\int_{\{f\leq f_0\}} \varphi^{s-4}e^{2(n-1)f}H(f) \nonumber
                 \\
                 &+\frac{1}{R^6}\int_{\{f\leq f_0\}} \varphi^{s-6}e^{2(n-1)f}
                 \\
                 \leq& \epsilon \int_{\{f\leq f_0\}} \varphi^se^{2(n-1)f}H^3(f)+CR^{-2\frac{2(n-1)+3b}{b} }\int_{\{f\leq f_0\}} \varphi^{ s-2\frac{2(n-1)+3b}{b}}.\nonumber
            \end{align}
            Substituting (3.18)-(3.21) into (3.12), we have
            \begin{align*}
                &\int_M\varphi^se^{2(n-1)f}H^3(f)
                \\
                \leq &\frac{C}{R^6}\int_{\{f\leq f_0\}\cup  \{f>f_0\}} \varphi^{s-6}e^{2(n-1)f}+\frac{C}{R^4}\int_{\{f\leq f_0\}\cup  \{f>f_0\}} \varphi^{s-4}e^{2(n-1)f}H(f) \nonumber
                \\
                &+\frac{C}{R^2}\int_{\{f\leq f_0\}\cup  \{f>f_0\}} \varphi^{s-2}e^{2(n-1)f}H^2(f)
                \\
                \leq & \epsilon \int_{\{f>f_0\}} \varphi^se^{2(n-1)f}H^3(f)+CR^{-2\frac{2(n-1)+3a}{a} }\int_{\{f>f_0\}} \varphi^{ s-2\frac{2(n-1)+3a}{a}} \\
                &+\epsilon \int_{\{f\leq f_0\}} \varphi^se^{2(n-1)f}H^3(f)+CR^{-2\frac{2(n-1)+3b}{b} }\int_{\{f\leq f_0\}} \varphi^{ s-2\frac{2(n-1)+3b}{b}}.
            \end{align*}
           Taking $s$ large enough, $\epsilon$ small enough and using the volume condition, we have
           \begin{align*}
               &\int_{B_R}e^{2(n-1)f}H^3(f) \leq CR^{2n+2-2\frac{2(n-1)+3a}{a} }+CR^{2n+2-2\frac{2(n-1)+3b}{b} }.
           \end{align*}
           Since $0<a,b<2$, hence
           \begin{align*}
              2n+2-2\frac{2(n-1)+3a}{a}<0,\quad  2n+2-2\frac{2(n-1)+3b}{b}<0.
           \end{align*}
           Letting $R\rightarrow \infty$, we have 
           \begin{align*}
               \int_{M}e^{2(n-1)f}H^3(f)=0.
           \end{align*}
           Thus, we conclude that $H(f)\equiv 0$, i.e. $F(u)\equiv0$. Hence $u$ is a psuedoharmonic function. By the gradient estimate of the positive psuedoharmonic function, $u$ is a constant (\cite{Re,BGM}). 
           {\qed}
           \par
           The proof of Theorem 1.1 is similar, we give a brief proof.

           ~\\
           $\mathbf{Proof\ of\ Theorem\ 1.1.}$ The proof is the same as in Theorem 1.2 until (3.12). Noting that $n=1$, then we have
           \begin{align}
                \int_M \varphi^s H^3(f)\leq  &\frac{C}{R^6}\int_M \varphi^{s-6}+\frac{C}{R^4}\int_M \varphi^{s-4}H(f) 
                +\frac{C}{R^2}\int_M \varphi^{s-2}H^2(f) \nonumber
                \\
                \leq &\epsilon \int_M \varphi^s H^3(f)+\frac{C}{R^6}\int_M \varphi^{s-6}.
            \end{align}
         Taking $\epsilon$ small enough, then Theorem 1.4 yields that
         \begin{align*}
             \int_{B_R} \varphi^s H^3(f)\leq C R^{-2}.
         \end{align*}
         Letting $R\rightarrow \infty$, we have $H(f)\equiv 0$. As above, we conclude that $u$ is a constant.

         \section{Proof of Theorem 1.6 and Theorem 1.7}
         In this section, we assume that $u\in C^1(M)$ is a positive weak solution of (1.7), i.e., $u>0$ and for every $0\leq \psi\in S^{1,2}(M)\cap L^{\infty}(M)$ with compact support, one has
             \begin{align}
                -\int_M\langle \nabla_b u, \nabla_b \psi \rangle+\int_M F(u)\psi  \leq 0,
            \end{align}
         where $F$ is a subcritical function with exponent $\sigma>1$. In order to prove Theorem 1.6, we establish the following lemma.
         \begin{lemma}
             Let $s>0$ large enough and $a\in (0,min\{\sigma-1,1\})$. If $F(t)>0 $ for any $ t>0 $, then there exists a constant $C>0$ such that for every function $ 0\leq \varphi\leq 1$ and $\varphi \in Lip(M)$ with compact support, we have
             \begin{align}
                \int_M u^{-\frac{\sigma}{\sigma-1}}F^{\frac{\sigma}{\sigma-1}}\varphi^s\leq&  C a^{-1-\frac{1-a}{2(\sigma-1)}}\left(\int_{M\backslash K} u^{-\frac{\sigma}{\sigma-1}}F^{\frac{\sigma}{\sigma-1}}\varphi^s\right)^{\frac{a+1}{2\sigma}} \nonumber
                \\
                &\cdot    \left(\int_M |\nabla_b\varphi|^{\frac{2(\sigma-a)}{\sigma-1}}\right)^{\frac{1}{2}}  \left(\int_M |\nabla_b\varphi|^{\frac{2\sigma}{\sigma-(a+1)}} \right)^{\frac{\sigma-(a+1)}{2\sigma}}           
            \end{align}
            and
            \begin{align}
                \left(\int_M u^{-\frac{\sigma}{\sigma-1}}F^{\frac{\sigma}{\sigma-1}}\varphi^s\right)^{1-\frac{a+1}{2\sigma}}\leq & C a^{-1-\frac{1-a}{2(\sigma-1)}}\left(\int_M |\nabla_b\varphi|^{\frac{2(\sigma-a)}{\sigma-1}}\right)^{\frac{1}{2}}  \nonumber
                \\
                &\cdot \left(\int_M |\nabla_b\varphi|^{\frac{2\sigma}{\sigma-(a+1)}} \right)^{\frac{\sigma-(a+1)}{2\sigma}},
            \end{align}
            where $K=\{x\in M| \varphi(x)=1\}$.
         \end{lemma}
         \begin{proof}
             It is easy to see that (4.3) can be derived directly from (4.2). Hence, we only need prove (4.2). Let $\psi=u^{-a}(u^{-\sigma}F)^b\varphi^s=u^{-a-b\sigma}F^b\varphi^s$ for any $\epsilon>0.$ Here $b>0$ is a constant to be determined later. Thus 
             \begin{align*}
                \nabla_b\psi=&-(a+b\sigma)u^{-a-b\sigma-1}F^b\varphi^s\nabla_bu+bu^{-a-b\sigma}F^{b-1}F^{'}\varphi^s\nabla_b u
                \\
                &+su^{-a-b\sigma}F^b\varphi^{s-1}\nabla_b\varphi \quad \quad a.e.\ \mathrm{on }\ M.
            \end{align*}
            Multiplying $\psi$ on both sides of (1.7), (4.1) implies that 
            \begin{align*}
                &(a+b\sigma)\int_M |\nabla_bu|^2 u^{-a-b\sigma-1}F^b\varphi^s -b\int_M |\nabla_bu|^2 u^{-a-b\sigma}F^{b-1}F^{'}\varphi^s 
                \\
                \leq& s\int_M  u^{-a-b\sigma}F^b\varphi^{s-1}\langle \nabla_bu,\nabla_b \varphi\rangle -\int_Mu^{-a-b\sigma}F^{b+1}\varphi^s.
            \end{align*}
            Noticing that $ F$ is subcritical with exponent $\sigma$, we obtain
            \begin{align}
           &a \int_M |\nabla_bu|^2 u^{-a-b\sigma-1}F^b\varphi^s +\int_Mu^{-a-b\sigma}F^{b+1}\varphi^s 
           \nonumber
           \\
           \leq& s\int_M  u^{-a-b\sigma}F^b\varphi^{s-1}\langle \nabla_bu,\nabla_b \varphi\rangle .
        \end{align}
        For the right hand side (4.4), using Young's inequality, we have
        \begin{align*}
          & s \int_M u^{-a-b\sigma}F^b\varphi^{s-1}\langle \nabla_bu,\nabla_b \varphi\rangle  
          \leq s \int_M |\nabla_bu| u^{-a-b\sigma}F^b\varphi^{s-1}|\nabla_b \varphi|  
        \\
        \leq & \int_M (u^{-\frac{(a+b\sigma+1)}{2}}F^{\frac{1}{2}b}|\nabla_bu|a^{\frac{1}{2}}\varphi^{\frac{1}{2}s})
        \cdot (a^{-\frac{1}{2}}u^{\frac{(a+b\sigma+1)}{2}-a-b\sigma}F^{\frac{b}{2}}\varphi^{\frac{s}{2}-1}s|\nabla_b\varphi| )
        \\
        \leq &\frac{1}{2}a \int_M |\nabla_bu|^2 u^{-a-b\sigma-1}F^b\varphi^s +\frac{C}{a}\int_M u^{-a-b\sigma+1}F^b\varphi^{s-2} |\nabla_b \varphi|^2 .
        \end{align*}
        Combing with (4.4), we get
        \begin{align}
            &\frac{a}{2}\int_M |\nabla_bu|^2 u^{-a-b\sigma-1}F^b\varphi^s  +\int_Mu^{-a-b\sigma}F^{b+1}\varphi^s
           \nonumber
           \\
           \leq& \frac{C}{a}\int_M u^{-a-b\sigma+1}F^b\varphi^{s-2} |\nabla_b \varphi|^2
           . 
        \end{align}
        Letting $b=\frac{1-a}{\sigma-1}$ and using Young's inequality, we obtain
        \begin{align}
            & \frac{C}{a}\int_M u^{-a-b\sigma+1}F^b\varphi^{s-2}|\nabla_b\varphi|^2=\frac{C}{a}\int_M u^{-\frac{1-a}{\sigma-1}}F^{\frac{1-a}{\sigma-1} }\varphi^{s-2}|\nabla_b\varphi|^2  \nonumber
           \\
           = & \int_M  u^{-\frac{1-a}{\sigma-1}}F^{\frac{1-a}{\sigma-1} } \varphi^{\frac{1-a}{\sigma-a}s}\cdot\frac{C}{a}\varphi^{\frac{\sigma-1}{\sigma-a}s-2}|\nabla_b\varphi|^2
           \\
           \leq & \frac{1}{2}\int_Mu^{-\frac{\sigma-a}{\sigma-1}}F^{\frac{\sigma-a}{\sigma-1}}\varphi^s+Ca^{-\frac{\sigma-a}{\sigma-1}}\int_M\varphi^{s-2\frac{\sigma-a}{\sigma-1}}|\nabla_b\varphi|^{2\frac{\sigma-a}{\sigma-1}}. \nonumber
        \end{align}
        From (4.5) and (4.6), we find that 
        \begin{align}
            &\frac{a}{2}\int_M |\nabla_bu|^2 u^{-\frac{2\sigma-a-1}{\sigma-1}}F^{\frac{1-a}{\sigma-1} }\varphi^s +\frac{1}{2}\int_Mu^{-\frac{\sigma-a}{\sigma-1}}F^{\frac{\sigma-a}{\sigma-1}}\varphi^s
           \nonumber
           \\
           \leq&Ca^{-\frac{\sigma-a}{\sigma-1}}\int_M\varphi^{s-2\frac{\sigma-a}{\sigma-1}}|\nabla_b\varphi|^{2\frac{\sigma-a}{\sigma-1}}
           \\
           \leq&Ca^{-\frac{\sigma-a}{\sigma-1}}\int_M|\nabla_b\varphi|^{2\frac{\sigma-a}{\sigma-1}}, \nonumber
        \end{align}
        since $0\leq \varphi\leq 1$ and $s $ is large enough.
        \par
        Let's consider another test function $\psi=(u^{-\sigma} F)^t \varphi^s$ for some $t>0$. Similarly, we have
        \begin{align}
            &\int_M u^{-\sigma t}F^{t+1}\varphi^s \leq s\int_M |\nabla_bu|u^{-\sigma t}F^t\varphi^{s-1}|\nabla_b\varphi|        \nonumber            \\          
            \leq & s\int_M|\nabla_bu|u^{-\frac{(a+b\sigma+1)}{2}}F^{\frac{1}{2}b}\varphi^{\frac{1}{2}s} \cdot u^{\frac{(a+b\sigma+1)}{2}-\sigma t}F^{t-\frac{1}{2}b}\varphi^{\frac{s}{2}-1}|\nabla_b\varphi| 
            \\
            \leq & C\left(\int_M |\nabla_bu|^2u^{-a-b\sigma-1}F^b\varphi^s \right)^{\frac{1}{2}}
            \cdot \left(\int_Mu^{(a+b\sigma+1)-2\sigma t}F^{2t-b}\varphi^{s-2} |\nabla_b\varphi|^2 \right)^{\frac{1}{2}}. \nonumber
        \end{align}
        Letting $t=\frac{1}{\sigma-1}$, we have the following estimates
        \begin{align}
            &\int_Mu^{(a+b\sigma+1)-2\sigma t}F^{2t-b}\varphi^{s-2} |\nabla_b\varphi|^2=\int_Mu^{-\frac{a+1}{\sigma-1}}F^{\frac{a+1}{\sigma-1}}\varphi^{s-2} |\nabla_b\varphi|^2  \nonumber
            \\
            =&\int_{M\backslash K}(u^{-\frac{a+1}{\sigma-1}}F^{\frac{a+1}{\sigma-1}}\varphi^{\frac{(a+1)s}{\sigma}}) \nonumber
            \cdot (\varphi^{\frac{\sigma-(a+1)}{\sigma}s-2}|\nabla_b\varphi|^2) \nonumber
            \\
            \leq & \left(\int_{M\backslash K} u^{-\frac{\sigma}{\sigma-1}}F^{\frac{\sigma}{\sigma-1}} \varphi^s\right)^{\frac{a+1}{\sigma}}
            \cdot\left( \int_{M\backslash K} \varphi^{s-\frac{2\sigma}{\sigma-(a+1)}}|\nabla_b\varphi|^{\frac{2\sigma}{\sigma-(a+1)}}\right)^{\frac{\sigma-a-1}{\sigma}}.
        \end{align}
        Substituting the estimate (4.7) and (4.9) into (4.8), we derive that
        \begin{align}
                \int_M u^{-\frac{\sigma}{\sigma-1}}F^{\frac{\sigma}{\sigma-1}}\varphi^s\leq&  C a^{-1-\frac{1-a}{2(\sigma-1)}}\left(\int_{M\backslash K} u^{-\frac{\sigma}{\sigma-1}}F^{\frac{\sigma}{\sigma-1}}\varphi^s\right)^{\frac{(a+1)}{2\sigma}} \nonumber
                \\
                &\cdot    \left(\int_M |\nabla_b\varphi|^{\frac{2(\sigma-a)}{\sigma-1}}\right)^{\frac{1}{2}}  \left(\int_M \varphi^{s-\frac{2\sigma}{\sigma-(a+1)}}|\nabla_b\varphi|^{\frac{2\sigma}{\sigma-(a+1)}} \right)^{\frac{\sigma-(a+1)}{2\sigma}}           
            \end{align}
            Letting $s>\frac{2\sigma}{\sigma-(a+1)}$, we get (4.2).
         \end{proof}
         Next, we give the proof of Theorem 1.6.

         ~\\
         $\mathbf{Proof\ of\ Theorem\ 1.6.}$ Let $h$ be a smooth function on such that
         $$  \left\{
        \begin{array}{rcl}
         &0\leq h(t)\leq 1,    & t\in [0,\infty); \\
         &h(t)=1,   &  t\in [0,1];  \\
          & h(t)=0 ,  & t\in[2,\infty]; \\
          & |h^{'}|\leq C,  & t\in[1,2].
        \end{array}
        \right.
        $$
        Define a sequence of functions $\{\varphi_i\}_{i\in \mathbb{N}}$ by
        \begin{align*}
             \varphi_i=i^{-1}\sum_{k=i+1}^{2i}h\left(\frac{r(x)}{2^k}\right),
         \end{align*}
         where $r(x)$ is the Carnot-Carath$\acute{\mathrm{e}}$odory distance from $x$ to a fixed point. Note that the support of $\left\{\nabla_bh\left(\frac{r(x)}{2^k}\right)\right\}_{k\in \mathbb{N}}$ are different from each other. Therefore, we have
         \begin{align}
             |\nabla\varphi_i|^\theta=i^{-\theta}\sum_{k=i+1}^{2i}|\nabla_bh\left(\frac{r(x)}{2^k}\right)|^\theta\leq C i^{-\theta}\sum_{k=i+1}^{2i} 2^{-k\theta} \chi_{\{2^k\leq r(\cdot)\leq 2^{k+1}\}},
         \end{align}
         where $ \chi$ is the characteristic function. Leaving $a=i^{-1}$ and $\varphi=\varphi_i $ in (4.3), we have
         \begin{align}
                \left(\int_M u^{-\frac{\sigma}{\sigma-1}}F^{\frac{\sigma}{\sigma-1}}\varphi_i^s\right)^{1-\frac{i^{-1}+1}{2\sigma}}\leq & C i^{1+\frac{1-i^{-1}}{2(\sigma-1)}}\left(\int_M |\nabla_b\varphi_i|^{\frac{2(\sigma-i^{-1})}{\sigma-1}}\right)^{\frac{1}{2}}  \nonumber
                \\
                &\cdot \left(\int_M |\nabla_b\varphi_i|^{\frac{2\sigma}{\sigma-(i^{-1}+1)}} \right)^{\frac{\sigma-(i^{-1}+1)}{2\sigma}},
            \end{align}
         According to (4.12), we need to estimate the following integral
         \begin{align*}
             J_i(\theta)=\int_M|\nabla_b\varphi_i|^\theta .
         \end{align*}
         For large $i$, using the volume condition, we have
         \begin{align}
             J_i(\theta)=&\int_M|\nabla_b\varphi_i|^\theta \leq Ci^{-\theta}\sum_{k=i+1}^{2i} 2^{-k\theta} V(2^{k+1}) \nonumber
             \\
             \leq &C i^{-\theta}\sum_{k=i+1}^{2i}2^{k(\frac{2\sigma}{\sigma-1}-\theta)}k^{\frac{1}{\sigma-1}}\leq C i^{\frac{1}{\sigma-1} -\theta}\sum_{k=i+1}^{2i}2^{k(\frac{2\sigma}{\sigma-1}-\theta)}.
         \end{align}
         Thus,
         \begin{align}
             &J_i\left( \frac{2(\sigma-i^{-1})}{\sigma-1}\right)\leq C i^{\frac{1}{\sigma-1} -\frac{2(\sigma-i^{-1})}{\sigma-1}}\sum_{k=i+1}^{2i}2^{\frac{2k}{i(\sigma-1)}}\leq Ci^{-\frac{\sigma}{\sigma-1}},
             \\
             &J_i\left( \frac{2\sigma}{\sigma-(i^{-1}+1)}\right) \leq C i^{\frac{1}{\sigma-1} -\frac{2\sigma}{\sigma-(i^{-1}+1)}}\sum_{k=i+1}^{2i}2^{\frac{2k}{i(\sigma-1)}}\leq Ci^{-\frac{\sigma}{\sigma-1}}.
         \end{align}
         From (4.12), (4.14) and (4.15), we obtain
         \begin{align}
             \left(\int_M u^{-\frac{\sigma}{\sigma-1}}F^{\frac{\sigma}{\sigma-1}}\varphi_i^s\right)^{1-\frac{i^{-1}+1}{2\sigma}}\leq & Ci^{1+\frac{1-i^{-1}}{2(\sigma-1)}}\cdot ( i^{-\frac{\sigma}{\sigma-1}})^{\frac{1}{2}+\frac{\sigma-(i^{-1}+1)}{2\sigma}} 
             \leq C. 
         \end{align}
         The inequality (4.16) implies that 
         \begin{align}
             \int_{B_{2^{i+1}}} u^{-\frac{\sigma}{\sigma-1}}F^{\frac{\sigma}{\sigma-1}}\leq C.
         \end{align}
         Letting $i \rightarrow \infty$, by Fatou’s lemma, we derive that 
         \begin{align}
             \int_M u^{-\frac{\sigma}{\sigma-1}}F^{\frac{\sigma}{\sigma-1}} \leq C.
         \end{align}
         Finally, from (4.2), we have
         \begin{align*}
             \int_{ B_{2^{i+1}}} u^{-\frac{\sigma}{\sigma-1}}F^{\frac{\sigma}{\sigma-1}} \leq  C \left(\int_{M\backslash B_{2^{i+1}}} u^{-\frac{\sigma}{\sigma-1}}F^{\frac{\sigma}{\sigma-1}}\right)^{\frac{(1+i^{-1})}{2\sigma}}.
         \end{align*}
         Letting $i \rightarrow \infty$, we conclude that
         \begin{align*}
             \int_{ M} u^{-\frac{\sigma}{\sigma-1}}F^{\frac{\sigma}{\sigma-1}}=0,
         \end{align*}
          which leads to a contradiction.
         {\qed} 
         
         ~\\
         $\mathbf{Proof\ of\ Theorem\ 1.7.}$ Letting $f=\mathrm{ln}\ u$, then we have
         \begin{align}
             \Delta_b f=\frac{\Delta_bu}{u}-|\nabla_b f|^2.
         \end{align}
         Let $\varphi$ be the cut-off function defined as in Section 3, i.e.
         $$  \left\{
        \begin{array}{rcl}
         &\varphi=1    &  in\ B_R, \\ 
         &0\leq \varphi\leq 1,    & in\ B_{2R}, \\
         &\varphi=0,   &  in\ M \backslash B_{2R} , \\
          & |\nabla_b \varphi|\leq \frac{C}{R},  & in\ M. \\
        \end{array}
        \right.
        $$
        Multiplying both sides of (4.19) by $\varphi^s$ for $s>2$, we have
        \begin{align}
            \int_M \varphi^s|\nabla_b f|^2=&\int_M \varphi^s\frac{\Delta_bu}{u}-\int_M\varphi^s\Delta_b f \nonumber
            \\
            \leq & s\int_M\varphi^{s-1}\langle \nabla_b f,\nabla_b\varphi\rangle\nonumber
            \\
            \leq& s\int_M\varphi^{s-1} |\nabla_b f|\cdot |\nabla_b \varphi|.
        \end{align}
        Here we use $u>0 $ and $\Delta_b u\leq 0$. Applying Young's inequality, we obtain
        \begin{align}
            \int_M \varphi^s|\nabla_b f|^2\leq &s\int_M\varphi^{s-1} |\nabla_b f|\cdot |\nabla_b \varphi| \nonumber
            \\
            \leq &\epsilon\int_M \varphi^s|\nabla_b f|^2+C\int_M \varphi^{s-2}|\nabla_b \varphi|^2.\nonumber
        \end{align}
        Choosing $\epsilon$ small enough and $R$ large enough, we derive that
        \begin{align}
            \int_M \varphi^s|\nabla_b f|^2\leq &  C\int_{B_{2R}} \varphi^{s-2}|\nabla_b \varphi|^2\leq \frac{C}{R^2}V(2R)\leq C.
        \end{align}
         Letting $R\rightarrow \infty$, we have
        \begin{align}
            \int_M |\nabla_b f|^2\leq C.
        \end{align}
        Returning to (4.20), we find that
        \begin{align}
            \int_M \varphi^s|\nabla_b f|^2\leq &s\int_M\varphi^{s-1} |\nabla_b f|\cdot |\nabla_b \varphi|
            \nonumber\\
            \leq& C \left(\int_{M\backslash B_R} \varphi^s|\nabla_b f|^2\right)^{\frac{1}{2}}\cdot \left( \int_M \varphi^{s-2}|\nabla_b \varphi|^2 \right)^{\frac{1}{2}}
            \nonumber\\
            \leq & C \left(\int_{M\backslash B_R} \varphi^s|\nabla_b f|^2\right)^{\frac{1}{2}}.
        \end{align}
        Thus 
        \begin{align*}
            \int_{B_R} |\nabla_b f|^2\leq C \left(\int_{M\backslash B_R} |\nabla_b f|^2\right)^{\frac{1}{2}}.
        \end{align*}
        Letting $R\rightarrow \infty$ again, we conclude that $|\nabla_bf|=0$ and $u$ is a constant.
        \\
        {\qed}

        \section{Data availability }
        This article does not make use of any data sets.

        \section{ Acknowledgments}
        The author thanks Professor Yuxin Dong and Professor Xiaohua Zhu for their continued support and encouragement.

\bibliographystyle{siam}
\bibliography{ref}

~\\
  Biqiang Zhao
  \\
  $Beijing\ International\ Center\ for$
  \\
  $Mathematical\ Research $
\\
  $ Peking\ University$
\\
   $Beijing\ 100871 ,$ $P.R.\ China $

\end{document}